\documentclass{amsproc}

\usepackage{amssymb, amscd, enumerate, colonequals, stmaryrd, mathdots, xcolor}
\usepackage{color}
\definecolor{chianti}{rgb}{0.6,0,0}
\definecolor{meretale}{rgb}{0,0,.6}
\definecolor{leaf}{rgb}{0,.35,0}
\usepackage[colorlinks=true, pagebackref, hyperindex, citecolor=meretale, urlcolor=leaf, linkcolor=chianti]{hyperref}

\newtheorem{theorem}{Theorem}[section]
\newtheorem{lemma}[theorem]{Lemma}
\newtheorem{proposition}[theorem]{Proposition}

\theoremstyle{definition}
\newtheorem{example}[theorem]{Example}
\newtheorem{remark}[theorem]{Remark}
\numberwithin{equation}{theorem}

\def\tr{{\operatorname{tr}}}
\def\inrev{\operatorname{in_{rev}}}
\def\image{\operatorname{im}}

\def\GL{\operatorname{GL}}

\def\SO{\operatorname{SO}}
\def\O{\operatorname{O}}

\def\Cl{\operatorname{Cl}}
\def\Pic{\operatorname{Pic}}
\def\res{\operatorname{res}}
\def\Der{\operatorname{Der}}
\def\Hom{\operatorname{Hom}}
\def\Spec{\operatorname{Spec}}

\def\fraka{\mathfrak{a}}
\def\frakb{\mathfrak{b}}
\def\frakm{\mathfrak{m}}
\def\frakp{\mathfrak{p}}

\def\AA{\mathbb{A}}
\def\DD{\mathbb{D}}
\def\GG{\mathbb{G}}
\def\UU{\mathbb{U}}
\def\VV{\mathbb{V}}
\def\WW{\mathbb{W}}
\def\ZZ{\mathbb{Z}}

\def\calK{\mathcal{K}}

\def\ge{\geqslant}
\def\le{\leqslant}
\def\bar{\overline}
\def\hat{\widehat}
\def\tilde{\widetilde}
\def\to{\longrightarrow}
\def\mapsto{\longmapsto}

\begin{document}
\title{The variety of nilpotent matrices is $F$-regular}

\author{Jack Jeffries}
\address{Department of Mathematics, University of Nebraska-Lincoln, 203 Avery Hall, Lincoln, NE~68588, USA}
\email{jack.jeffries@unl.edu}

\author{Vaibhav Pandey}
\address{Department of Mathematics, Indian Institute of Technology Madras, Chennai, Tamil Nadu~600036, India}
\email{vaibhavpandey@iitm.ac.in}

\author{Anurag K. Singh}
\address{Department of Mathematics, University of Utah, 155 South 1400 East, Salt Lake City, UT~84112, USA}
\email{singh@math.utah.edu}

\thanks{J.J. was supported by NSF CAREER award DMS 2044833 and A.K.S. by NSF grants DMS~2101671 and DMS~2349623.}


\begin{abstract}
We give an elementary proof that the coordinate ring of the variety of nilpotent matrices is $F$-regular; over an infinite field $K$, this ring also arises as the nullcone for the conjugation action of the general linear group $\textrm{GL}_n(K)$ on the polynomial ring $K[X]$, where~$X$ is an $n\times n$ matrix of indeterminates. We prove that the divisor class group of the coordinate ring is the cyclic group~$\ZZ/n\ZZ$.

We then study the case of symmetric nilpotent matrices, where the picture is completely different: the coordinate ring is not normal for $n\geqslant 2$; for $K$ algebraically closed of characteristic other than two, we prove that the coordinate ring is an integral domain precisely when $n$ is odd.
\end{abstract}
\maketitle

\section{Nilpotent matrices}

Set $S\colonequals K[X]$, where $K$ is a field, and
\[
X\colonequals\begin{bmatrix}
x_{11} & x_{12} & \hdots & x_{1n}\\
x_{21} & x_{22} & \hdots & x_{2n}\\
\vdots & \vdots & & \vdots \\
x_{n1} & x_{n2} & \hdots & x_{nn}
\end{bmatrix}
\]
is an $n\times n$ matrix of indeterminates. The characteristic polynomial of $X$ takes the form
\[
\det(tI-X)\ =\ t^n - c_1t^{n-1} + c_2t^{n-2} - \dots +(-1)^nc_n,
\]
where, for example, $c_1$ is the trace of the matrix $X$, and $c_n$ its determinant. The algebraic set defined by the vanishing of the polynomials $c_1,c_2,\dots,c_n$ in $\AA^{n\times n}_K$ is precisely the set of nilpotent~$n\times n$ matrices. We prove:

\begin{theorem}
\label{theorem:nilpotent}
Let $X$ be an $n\times n$ matrix of indeterminates over a field $K$. Then~$R\colonequals K[X]/(c_1,\dots,c_n)$ is a normal complete intersection ring, with divisor class group $\ZZ/n\ZZ$. If~$K$ has characteristic zero, the ring $R$ has rational singularities; if $K$ has positive characteristic, $R$ is $F$-regular.
\end{theorem}

The fact that $R$ has rational singularities was previously known by other methods \cite{KP, Donkin, Mehta:vdK, Weyman}; see also Remark~\ref{remark:history}.

\begin{remark}
When $K$ is an infinite field, the ring $R$ in Theorem~\ref{theorem:nilpotent} also arises as a nullcone: for a group $G$ acting on a polynomial ring $S$ via degree-preserving~$K$-algebra automorphisms, the \emph{nullcone} of the action is the ring~$S/(\frakm_{S^G}S)$, where $\frakm_{S^G}$ is the homogeneous maximal ideal of $S^G$. Set~$S\colonequals K[X]$ for~$X$ an $n\times n$ matrix of indeterminates, and consider the general linear group~$\GL_n(K)$ acting on $S$ by conjugation, i.e., take the action determined~by 
\[
M\colon X\mapsto MXM^{-1}
\]
where $M\in \GL_n(K)$. Then the invariant ring is generated by the coefficients of the characteristic polynomial of $X$, i.e.,
\begin{equation}
\label{equation:conjugation:invariant}
S^{\GL_n(K)}\ =\ K[c_1,c_2,\dots,c_n].
\end{equation}
While this is certainly well-known, we sketch a proof following \cite[Example~2.1.3]{Derksen:Kemper}, as we require a modification in the next section.

The proof reduces to the case where the field $K$ is algebraically closed. Since the characteristic polynomial of $X$ is fixed under conjugation, it is clear that the $c_i$ are invariant. The symmetric group $\Sigma_n$ may be viewed as the subgroup of permutation matrices; upon specializing the off-diagonal entries of $X$ to zero, $c_1,\dots,c_n$ specialize to the elementary symmetric polynomials $e_1,\dots,e_n$ in the diagonal entries $x_{11},\dots,x_{nn}$. The action of $\GL_n(K)$ on $S$ is a right action, corresponding to a left action on the affine space $\AA$ of $n\times n$ matrices over $K$. Let~$\DD\subset\AA$ be the set of diagonal matrices, and $\WW\subset\AA$ the set of matrices with distinct eigenvalues; note that $\WW$ is Zariski open --- its complement is defined by the vanishing of the discriminant of~$\det(tI-X)$. Since a matrix with distinct eigenvalues is diagonalizable, $\WW$ is contained in the image of $\DD$ under the $\GL_n(K)$ action. The restriction of an invariant $f\in S^{\GL_n(K)}$ to $\DD$ is $\Sigma_n$-invariant, i.e., there exists a polynomial $p(e_1,\dots,e_n)$ that agrees with $f$ on~$\DD$. But then the polynomial
\[
f-p(c_1,\dots,c_n)
\]
vanishes on $\DD$, and hence also on the dense subset $\WW\subset\AA$. It follows that $f$ equals~$p(c_1,\dots,c_n)$, proving~\eqref{equation:conjugation:invariant}.

Nullcones have been studied extensively in the last some years; in \cite{HJPS} it is proven that for various classical group actions, the irreducible components of the nullcones are Cohen-Macaulay rings, while their $F$-regularity is investigated in~\cite{Lorincz:ASENS, Lorincz:symplectic, PTW}. Theorem~\ref{theorem:nilpotent} fits into this body of work: it says precisely that the nullcone for the conjugation action of $\GL_n(K)$ is $F$-regular. In contrast, consider an algebraically closed field $K$ of characteristic other than two, $S\colonequals K[X]$ for $X$ a \emph{symmetric} $n\times n$ matrix of indeterminates where $n\ge 2$, and the orthogonal group $\O_n(K)$ acting on $S$ via conjugation. Then the nullcone is not normal; see Theorem~\ref{theorem:sym}.
\end{remark}

\begin{proof}[Proof of Theorem~\ref{theorem:nilpotent}]
We first claim that $c_1,\dots,c_n$ form a regular sequence in $S\colonequals K[X]$. For this, it suffices to verify that the off-diagonal entries of~$X$, along with $c_1,\dots,c_n$, form a system of parameters for the polynomial ring~$S$, i.e., that
\[
K[x_{11},x_{22},\dots,x_{nn}]/(\bar{c}_1,\dots,\bar{c}_n)
\]
is Artinian, where $\bar{c}_i$ denotes the image of $c_i$ in $S/(x_{ij}:i\neq j)$. This is clear since $\bar{c}_1,\dots,\bar{c}_n$ are precisely the elementary symmetric polynomials in $x_{11},x_{22},\dots,x_{nn}$. An alternative argument that the elements $c_1,\dots,c_n$ form a regular sequence in $S$ appears later within this proof.

We next prove that $R$ has rational singularities, or that $R$ is $F$-regular; for this, it suffices to assume that $K$ is algebraically closed, of characteristic $p>0$, and to prove that $R$ is $F$-regular --- by \cite[Theorem~4.3]{Smith:ratsing}, it then follows that $R$ has rational singularities when $K$ has characteristic zero. Towards this, we first prove that $R$ is $F$-pure using Fedder's criterion, \cite[Theorem~1.12]{Fedder}; since $c_1,\dots,c_n$ is a regular sequence in $S$, one needs to verify that
\begin{equation}
\label{equation:fedder}
(c_1 \cdots c_n)^{p-1}\ \notin\ \frakm^{[p]},
\end{equation}
where $\frakm$ is the homogeneous maximal ideal of $S$.

Consider the monomial order
\[
x_{11} > x_{12} > \cdots > x_{1n} > x_{21} > \cdots > x_{nn}
\]
and the induced revlex term order on $S$. The leading monomial of each $c_i$ is the product of the indeterminates on the $i\textsuperscript{th}$ antidiagonal, namely
\[
\inrev(c_i)\ =\ x_{1i}\, x_{2,i-1}\, \cdots\, x_{i1}.
\]
It follows that $(c_1 \cdots c_n)^{p-1}$ has leading monomial
\[
\prod_{i+j\le n+1} x_{ij}^{p-1}
\]
so~\eqref{equation:fedder} indeed holds.

We use Glassbrenner's criterion \cite[Theorem~3.1]{Glassbrenner} to show that $R$ is $F$-regular. Set $\Delta$ to be the bottom right size $n-1$ minor of $X$. We first verify that $\Delta$ is not in any minimal prime of the ideal $(c_1,\dots,c_n)S$. For this, it suffices to note that the elements $c_1,\dots,c_n,\Delta$ form a regular sequence in $S$. Under the revlex term order above, the leading monomial of $\Delta$ is the product of the indeterminates on the $(n+1)\textsuperscript{st}$ antidiagonal, i.e.,
\[
\inrev(\Delta)\ =\ x_{2n}\, x_{3,n-1}\, \cdots\, x_{n2}.
\]
It follows that the leading monomials of $c_1,\dots,c_n,\Delta$ form a regular sequence in $S$, and hence that $c_1,\dots,c_n,\Delta$ is a regular sequence.

Next, note that the algebraic set $\VV\colonequals\VV(c_1,\dots,c_n)$ consists precisely of~$n\times n$ matrices all of whose eigenvalues are zero. Considering their Jordan form, for example, matrices in $\VV$ have rank at most $n-1$. The group $G\colonequals\GL_n(K)$ acts on $\VV$ by conjugation; the matrices in $\VV$ with rank exactly $n-1$, i.e., the \emph{regular nilpotent matrices}, constitute a single orbit, as they all have the same Jordan form, and hence form a smooth subset. Those for which the bottom right $n-1$ minor is invertible form an open subset, that is also smooth; as $R$ is reduced, it follows that $R_\Delta$ is regular. Hence some power of $\Delta$ is a test element for the ring $R$ by~\cite[Theorem~3.4]{HH:strong}. But the test ideal of~$R$ is radical since $R$ is $F$-pure, \cite[Proposition~2.5]{Fedder:Watanabe}, so~$\Delta$ is a test element. By Glassbrenner's criterion, it thus suffices to verify that
\[
\Delta(c_1 \cdots c_n)^{p-1}\ \notin\ \frakm^{[p]}.
\]
This is straightforward considering, once again, the leading monomial under the revlex term order; this completes the proof that $R$ is $F$-regular.

For the calculation of the divisor class group, assume first that $K$ is algebraically closed. Consider once again the conjugation action of $G$ on the set $\VV$ of nilpotent matrices; the orbits correspond to partitions of $n$ along Jordan blocks. The regular nilpotent matrices constitute a single orbit, say $\UU$, that has dimension $n^2-n$, while other orbits have strictly smaller dimension: the partition $(n-1,1)$ corresponds to an orbit of dimension $n^2-n-2$, while other orbits have yet smaller dimension; see Remark~\ref{remark:orbits}. It follows that $\VV\smallsetminus\UU$ has no components of codimension one, and hence that
\[
\Cl(R)\ \cong\ \Cl(\UU).
\]
Since $\UU$ is smooth, Weil divisors and Cartier divisors on $\UU$ coincide, so
\[
\Cl(\UU)\ \cong\ \Pic(\UU).
\]
The orbit $\UU$ may be viewed as $G/H$, where $H$ is the stabilizer in $G$ of a nilpotent Jordan block $J$. By Remark~\ref{remark:orbits},
\[
H\ =\ \{a_0\mathbf{1} + a_1J + a_2J^2 + \dots + a_{n-1}J^{n-1}\mid a_i\in K,\ a_0\neq 0\}.
\]

Let $\chi(-)$ denote the \emph{algebraic character group} of an algebraic group, i.e.,
\[
\chi(-)\colonequals\Hom_{\textrm{alg.group}}(-,\GG_m).
\]
Since $\Pic(G)=0$, by \cite[Theorem~4]{Popov} and the surrounding discussion, one has an exact sequence
\begin{equation}
\label{equation:popov}
\CD
0 @>>> \chi(G) @>\res>> \chi(H) @>>> \Pic(G/H) @>>> 0,
\endCD
\end{equation}
where $\res$ denotes the restriction map. Note that $\chi(G)$ is infinite cyclic, being generated by the determinant character. The group $\chi(H)$ is also infinite cyclic, since each algebraic character factors through $H/H_u$, where
\[
H_u\colonequals \mathbf{1}+JK[J]
\]
is the unipotent radical of $H$; specifically, $H/H_u\cong\GG_m$, and $\chi(H)$ is generated by the character
\[
\Psi\colon a_0\mathbf{1} + a_1J + a_2J^2 + \dots + a_{n-1}J^{n-1}\ \mapsto\ a_0.
\]
The determinant character restricts to $H$ with
\[
\det(a_0\mathbf{1} + a_1J + a_2J^2 + \dots + a_{n-1}J^{n-1})\ =\ a_0^n,
\]
so the exact sequence~\eqref{equation:popov} takes the form
\[
\CD
0 @>>> \ZZ @>n>> \ZZ @>>> \Pic(G/H) @>>> 0,
\endCD
\]
completing the proof when $K$ is algebraically closed. 

To extend the calculation of the divisor class group to the case of an arbitrary field, it suffices to note that the order $n$ line bundle on the space $\UU$ of regular nilpotent matrices is defined independent of the field; it is simply the bundle $\calK$ that associates a regular nilpotent matrix to its kernel:

The kernel of the nilpotent Jordan block $J$ is spanned by the basis vector~$e_1$. Let $K_\Psi$ denote the field $K$ viewed as a one-dimensional representation of $H$ with
\[
hz\colonequals \Psi(h)z
\]
for $h\in H$ and $z\in K_\Psi$. We claim that the kernel bundle $\calK$ has total space $G\times_H K_\Psi$, namely~$G\times K_\Psi$ under the equivalence relation
\[
(g,\, z)\ \sim\ (gh,\, \Psi(h)^{-1}z).
\]
Identifying $G/H$ with $\UU$ via $gH\mapsto gJg^{-1}$, one has
\begin{alignat*}3
G\times_H K_\Psi\ & \to\ \qquad \calK\\
(g,\, z)\qquad & \mapsto\ (gJg^{-1},\, zge_1);
\end{alignat*}
note that this is indeed well-defined at the level of equivalence classes:
\[
(gh,\, \Psi(h)^{-1}z)\ \mapsto\ (ghJh^{-1}g^{-1},\, \Psi(h)^{-1}zghe_1)\ =\ (gJg^{-1},\, zge_1),
\]
since $h$ commutes with $J$ and $he_1=\Psi(h)e_1$.
\end{proof}

\begin{remark}
Let $X$ and $R$ be as in Theorem~\ref{theorem:nilpotent}, and let $X'$ denote the matrix obtained by deleting the first row of $X$. We claim that the ideal $\frakp\colonequals I_{n-1}(X')R$ is a height one prime that generates the divisor class group of $R$.

We first verify that the ideal $\frakp$ is radical. Consider its preimage in $S$,
\[
\fraka \colonequals (c_1,\dots,c_{n-1}) + I_{n-1}(X').
\]
With the revlex term order from the proof of Theorem~\ref{theorem:nilpotent}, the leading monomials of $c_1,\dots,c_{n-1}$ involve indeterminates that are disjoint from each other, and also from the indeterminates that occur in the leading monomials of $I_{n-1}(X')$. Set $\frakb$ to be the ideal of $S$ generated by these leading monomials. The ring $S/\inrev(I_{n-1}(X'))$ is Cohen-Macaulay by \cite[Theorem~4.3.2 and Remark~4.3.5]{BCRV}, so $S/\frakb$ is Cohen-Macaulay as well. Since $I_{n-1}(X')S$ has height $2$, and $\frakb \subseteq \inrev(\fraka)$, one has
\[
n^2-n-1\ =\ \dim S/\frakb\ \ge\ \dim S/\fraka\ \ge \ n^2-n-1,
\]
so equality holds in the above display.

The rings $S/I_{n-1}(X')$ and $S/\inrev(I_{n-1}(X'))$ have the same multiplicity; killing the regular sequences $c_1,\dots,c_{n-1}$ and $\inrev(c_1),\dots,\inrev(c_{n-1})$ in the respective rings, it follows that $S/\fraka$ and $S/\frakb$ have the same multiplicity, and hence that the rings $S/\inrev(\fraka)$ and $S/\frakb$ have the same multiplicity. As $S/\frakb$ is Cohen-Macaulay, additivity of multiplicities yields $\frakb=\inrev(\fraka)$. Since $\frakb$ is radical, it follows that $\fraka$ is radical, and hence that $\frakp$ is a radical ideal of $R$.

To see that $\VV(\frakp)$ is an irreducible algebraic set, it suffices to verify that its dense subset $\VV(\frakp)\cap\UU$ is irreducible. A regular nilpotent matrix $A$ belongs to $\VV(\frakp)$ precisely if the matrix $A'$ obtained by deleting the first row has rank $n-2$; as~$A$ has rank $n-1$, this is equivalent to $e_1 \in \image(A)$. Writing $A=BJB^{-1}$, with $J$ the nilpotent Jordan block, $\image(A) = \image[b_1\, \dots\, b_{n-1}]$, where $b_i$ is the $i\textsuperscript{th}$ column of $B$. Then $e_1\in\image(A)$ if and only if $\det[b_1\, \dots\, b_{n-1}\, e_1] = 0$; this defines an irreducible subset $\WW$ of $\GL_n(K)$, with $\VV(\frakp)\cap\UU$ being the image of $\WW$ under the map
\begin{alignat*}3
\GL_n(K)\ & \to\ \quad \UU\\
B\quad & \mapsto\ BJB^{-1}.
\end{alignat*}
Since $\WW$ is irreducible, so is its image $\VV(\frakp)\cap\UU$. This completes the argument that~$\frakp$ is a prime ideal of $R$.

For a regular nilpotent matrix $A$, one has $\ker(A^{n-1}) = \image(A)$, so the condition that $A\in\VV(\frakp)$ is equivalent to the condition that $A^{n-1}e_1=0$. Hence
\begin{alignat*}3
\UU\ & \to \qquad \calK\\
A\ & \mapsto\ (A,\ A^{n-1} e_1)
\end{alignat*}
is a regular section, whose vanishing locus is precisely $\VV(\frakp)\cap\UU$. It follows that the height one prime ideal $\frakp$ corresponds to the line bundle $\calK$.

Since the ring $R$ is $F$-regular/of $F$-regular type, each torsion element of $\Cl(R)$ is a maximal Cohen-Macaulay $R$-module: when the characteristic of $K$ is coprime to the order of the torsion element, this is \cite[Corollary~2.9]{Watanabe:dim2}, while the characteristic assumption is removed in \cite[Theorem~C]{Carvajal-Rojas}. Under the weaker assumption that a ring has rational singularities, it is no longer true that torsion elements of the divisor class group are maximal Cohen-Macaulay modules, \cite[Theorem~1.1]{Singh:cyclic}.

In conclusion, the prime ideal $\frakp$ and all its symbolic powers---equivalently, all rank one reflexive $R$-modules---are maximal Cohen-Macaulay $R$-modules.
\end{remark}

\begin{remark}
\label{remark:orbits}
Let $N$ be an $n\times n$ matrix over a field $K$. If the minimal polynomial of $N$ equals its characteristic polynomial, then the centralizer of $N$ in the matrix algebra~$\mathrm{Mat}_n(K)$ is
\[
K[N]\ =\ \mathrm{span}_K\{ \mathbf{1}, N, N^2, \dots, N^{n-1}\}.
\]
In particular, the centralizer of an $n\times n$ nilpotent Jordan block $N$ has dimension~$n$, corresponding to the orbit of $N$ having dimension $n^2-n$.

More generally, let $N$ be a nilpotent matrix with Jordan blocks of size $\lambda_1,\dots,\lambda_k$. Then the orbit of $N$ has dimension
\[
n^2-\sum(\lambda'_i)^2,
\]
where $(\lambda'_1,\dots,\lambda'_\ell)$ is the \emph{conjugate partition} of $n$, i.e., the partition obtained by reflecting the Ferrers diagram of $(\lambda_1,\dots,\lambda_k)$; see \cite[Corollary~6.1.4]{Collingwood:McGovern}.

For example, the partition $(n)$ has conjugate partition $(1,\dots,1)$, with the corresponding orbit having dimension $n^2-n$; this is the orbit of largest dimension, consisting of the regular nilpotent matrices. The next largest dimension comes from the partition $(n-1,1)$, that has conjugate partition $(2,1,\dots,1)$, with the corresponding orbit having dimension $n^2-(4+1+\dots+1) = n^2-n-2$.
\end{remark}

\begin{remark}
\label{remark:history}
For any partition of $n$, the closure of the corresponding orbit is known to have rational singularities: this is due to Kraft and Procesi \cite{KP} in characteristic zero, and to Donkin \cite{Donkin} in characteristic $p>0$. Alternative approaches can be found in \cite{Mehta:vdK, Weyman}.
\end{remark}

\section{Symmetric nilpotent matrices}

Let $X$ be a symmetric $n\times n$ matrix of indeterminates over a field $K$, and
\[
\det(tI-X)\ =\ t^n - c_1t^{n-1} + c_2t^{n-2} - \dots +(-1)^nc_n
\]
its characteristic polynomial. Our focus in this section is on the coordinate ring of the algebraic set of symmetric nilpotent matrices, $R\colonequals K[X]/(c_1,\dots,c_n)$, with Theorem~\ref{theorem:sym} summarizing our main results. We first record that when $K$ is infinite, the ring $R$ is indeed the nullcone for the conjugation action of the orthogonal group
\[
\O_n(K)\colonequals\{Q\in \GL_n(K)\, | \, Q^\tr Q = \mathbf{1}\}
\]
on the polynomial ring $K[X]$.

\begin{proposition}
\label{proposition:symmetric:nullcone}
Let $X$ be a symmetric $n\times n$ matrix of indeterminates over an infinite field $K$. Consider the action of the orthogonal group $\O_n(K)$ on $S\colonequals K[X]$ determined by 
\[
Q\colon X\mapsto QXQ^\tr
\]
where $Q\in\O_n(K)$. Then the invariant ring is generated by the coefficients of the characteristic polynomial of $X$, i.e.,
\[
S^{\O_n(K)}\ =\ K[c_1,c_2,\dots,c_n].
\]
\end{proposition}

\begin{proof}
The argument resembles the one presented for~\eqref{equation:conjugation:invariant}. Take $\AA$ to be the affine space of $n\times n$ symmetric matrices, $\DD$ the diagonal matrices, and $\WW\subset\AA$ the open set of symmetric matrices having distinct eigenvalues. The permutation matrices $\Sigma_n$ form a subgroup of $\O_n(K)$ that acts on $\DD$. While similar symmetric matrices need not be in the same $\O_n(K)$ orbit when $K$ has characteristic two, see Example~\ref{example:char:two}, it is nonetheless true that $\WW$ is contained in the image of $\DD$ under the action of $\O_n(K)$ by Lemma~\ref{lemma:orthogonal:char:two}.
\end{proof}

We now prove the main result of this section:

\begin{theorem}
\label{theorem:sym}
Fix an integer $n\ge 2$. Let $X$ be a symmetric $n\times n$ matrix of indeterminates over a field $K$, and
\[
\det(tI-X)\ =\ t^n - c_1t^{n-1} + c_2t^{n-2} - \dots +(-1)^nc_n
\]
its characteristic polynomial. Set~$R\colonequals K[X]/(c_1,\dots,c_n)$. Then:
\begin{enumerate}[\quad\rm(1)]
\item The ring $R$ is a complete intersection;
\item The ring $R$ is not normal;
\item If $K$ has characteristic other than two, then $R$ is reduced; it is an integral domain if $n$ is odd, and has two minimal primes if $n$ is even and $K$ is algebraically closed.
\item If $K$ has characteristic two, then $R$ is not reduced.
\end{enumerate}
\end{theorem}

\begin{proof}
(1) The fact that $c_1,\dots,c_n$ form a regular sequence in $S\colonequals K[X]$ follows from the proof of Theorem~\ref{theorem:nilpotent}: if $Z$ is an $n\times n$ generic matrix, the coefficients of its characteristic polynomial, along with $z_{ij}-z_{ji}$ and $z_{ij}$, for $i<j$, form a regular sequence in $K[Z]$.

For the rest of the proof, we assume that the field $K$ is algebraically closed. For~(2), use the Jacobian criterion; since $c_n$ is the determinant of the symmetric matrix $X$, it follows that
\[
\frac{\partial c_n}{\partial x_{ij}}\ \in\ I_{n-1}(X)
\]
for each $i,j$, where $I_{n-1}(X)$ is the ideal of $S$ generated by the size $n-1$ minors of~$X$. Hence the defining ideal of the singular locus of $R$ is contained in $I_{n-1}(X)R$. To verify that $R$ is not normal, it suffices to check that $I_{n-1}(X)R$ has height at most one in $R$, equivalently that
\[
I_{n-1}(X) + (c_1,\dots,c_n)S\ =\ I_{n-1}(X) + (c_1,\dots,c_{n-2})S
\]
has height at most $n+1$ in $S$. But $I_{n-1}(X)$ has height $3$, which settles~(2).

Towards (3), assume that $K$ has characteristic other than two; we first prove that $R$ is reduced. Since $R$ is a complete intersection ring, it suffices to check that it satisfies the Serre condition $(R_0)$. Working in the polynomial ring $S$, set
\[
D\colonequals \det
\begin{bmatrix}
\frac{\partial c_1}{\partial x_{11}} & \frac{\partial c_2}{\partial x_{11}} & \hdots & \frac{\partial c_n}{\partial x_{11}}\\[1ex]
\frac{\partial c_1}{\partial x_{12}} & \frac{\partial c_2}{\partial x_{12}} & \hdots & \frac{\partial c_n}{\partial x_{12}}\\[1ex]
\vdots & \vdots & & \vdots \\[1ex]
\frac{\partial c_1}{\partial x_{1n}} & \frac{\partial c_2}{\partial x_{1n}} & \hdots & \frac{\partial c_n}{\partial x_{1n}}\\[1ex]
\end{bmatrix}.
\]
As $D$ is a maximal minor of the Jacobian matrix, the ring $R_D$ is regular; it suffices to verify that $D$ is not in any minimal prime of the ideal $(c_1,\dots,c_n)S$.

For integers $1 \le a_1 < a_2 < \dots < a_k \le n$ and $1 \le b_1 < b_2 < \dots < b_k \le n$, set
\[
[a_1 \dots a_k\mid b_1 \dots b_k]
\]
to be the determinant of the submatrix of~$X$ with rows $a_1,\dots,a_k$ and columns $b_1,\dots,b_k$. With this notation, one has
\[
c_k\ =\ \sum [a_1 \dots a_k\mid a_1 \dots a_k],
\]
where the summation runs over all $1 \le a_1 < a_2 < \dots < a_k \le n$. Note that
\[
\frac{\partial}{\partial x_{11}} [a_1 \dots a_k\mid a_1 \dots a_k] \ =\
\begin{cases}
[a_2 \dots a_k\mid a_2 \dots a_k] & \text{ if }a_1=1,\\
0 & \text{ else},
\end{cases}
\]
and that for $j\ge 2$ one has
\[
\frac{\partial}{\partial x_{1j}} [a_1 \dots a_k\mid a_1 \dots a_k]\ =\
\begin{cases}
\pm 2 [a_2 \dots a_k\mid a_1 \dots \hat{a_i} \dots a_k] & \text{ if } a_1=1 \text{ and } j=a_i,\\
0 & \text{ else}.
\end{cases}
\]
Working modulo the ideal $(x_{ij} \mid i+2\le j)S$, it follows that
\[
D\ \equiv\ \det
\begin{bmatrix}
1 & * & * & \hdots & * \\
0 & \pm 2 x_{12} & * & \hdots & *\\
0 & 0 & \pm 2x_{12}x_{23} & & *\\
\vdots & \vdots & & & \\
0 & 0 & 0 & & \pm 2x_{12}x_{23}\cdots x_{n-1,n}\\
\end{bmatrix}.
\]
We claim that $D$ is not in any minimal prime of the ideal
\[
\fraka\colonequals (c_1,\dots,c_n)S + (x_{ij} \mid i+2\le j)S,
\]
for which it suffices to check that 
\[
x_{12}\, x_{23}\, \cdots\, x_{n-1,n}
\]
is not in any minimal prime of $\fraka$. But this is clear since $S/\fraka$ is a complete intersection ring with system of parameters $x_{12},x_{23},\dots,x_{n-1,n}$. This completes the proof that the ring $R$ is reduced.

Set $\Delta$ to be the bottom right size $n-1$ minor of $X$; we claim that $c_1,\dots,c_n,\Delta$ form a regular sequence in $S$. The proof is analogous to that used for a generic matrix in Theorem~\ref{theorem:nilpotent}, namely consider the monomial order
\[
x_{11} > x_{12} > \cdots > x_{1n} > x_{22} > x_{23} > \cdots > x_{2n} > \cdots > x_{nn}
\]
and the induced revlex term order on $S$. The leading monomial of each $c_i$ is again the product of the indeterminates on the $i\textsuperscript{th}$ antidiagonal, while the leading monomial of $\Delta$ is the product of the indeterminates on the $(n+1)\textsuperscript{st}$ antidiagonal; the claim follows. Set $\frakp\colonequals I_{n-1}(X)$, and note that $\Delta\in\frakp$.

Let $\AA$ denote affine space corresponding to the indeterminates in the matrix~$X$. To complete the proof of (3), it suffices to determine the number of irreducible components of the algebraic set $\VV\colonequals\VV(c_1,\dots,c_n)$ in~$\AA$. This is the same as the number of irreducible components of the dense subset 
\[
\VV\smallsetminus\VV(\frakp)
\]
consisting of $n\times n$ symmetric regular nilpotent matrices. The group $\SO_n(K)$ acts on this set by conjugation, with the number of irreducible components being the number of orbits; the assertion now follows from Lemma~\ref{lemma:orthogonal:orbits}.

It remains to prove (4). Consider a partial derivative
\[
\frac{\partial c_k}{\partial x_{ij}}.
\]
Since $S$ has characteristic two, it is readily seen that this partial derivative vanishes if $i<j$. It follows that the defining ideal of the singular locus of $R$ is generated by the size $n$ minor 
\[
\Delta\colonequals\det\left[\frac{\partial c_k}{\partial x_{ii}}\right]
\]
of the Jacobian matrix. We claim that $\Delta$ is an element of the subring $K[c_1,\dots,c_n]$ of $S$. Since $n\ge 2$, the claim implies that $\Delta\in(c_1,\dots,c_n)S$, so the image of $\Delta$ in~$R$ is zero and the singular locus of $R$ is $\Spec R$. It follows that $R$ is not reduced.

Consider the conjugation action of the orthogonal group $\O_n(K)$ on $S$. Given a matrix $Q=(q_{ij})$ in $\O_n(K)$, one has
\[
\left(QXQ^\tr\right)_{ii}\ =\ \sum_{j\le k} q_{ij}x_{jk}q_{ik} + \sum_{j>k} q_{ij}x_{kj}q_{ik}.
\]
Since $K$ has characteristic two, it follows that for $j<k$, the coefficient of $x_{jk}$ in the display above is
\[
q_{ij}q_{ik} + q_{ik}q_{ij}\ =\ 0,
\]
so
\[
\left(QXQ^\tr\right)_{ii}\ =\ \sum_j q_{ij}^2x_{jj}.
\]
Specifically, under conjugation by $Q$, one has
\[
\begin{bmatrix}x_{11}\\[1ex] x_{22}\\[1ex] \vdots \\[1ex] x_{nn}\end{bmatrix}
\ \mapsto\
\begin{bmatrix}
q_{11}^2 & q_{12}^2 & \hdots & q_{1n}^2\\[1ex]
q_{21}^2 & q_{22}^2 & \hdots & q_{2n}^2\\[1ex]
\vdots & \vdots & & \vdots \\[1ex]
q_{n1}^2 & q_{n2}^2 & \hdots & q_{nn}^2
\end{bmatrix}
\begin{bmatrix}
x_{11}\\[1ex] x_{22}\\[1ex] \vdots \\[1ex] x_{nn}
\end{bmatrix}
\ =\ \tilde{Q}
\begin{bmatrix}x_{11}\\[1ex] x_{22}\\[1ex] \vdots \\[1ex] x_{nn}\end{bmatrix},
\]
where $\tilde{Q}\colonequals (q^2_{ij})$. Note that $\tilde{Q}\in\O_n(K)$ since $K$ has characteristic two. Quite generally, if a group $G$ acts $K$-linearly on a polynomial ring $S$, there is a compatible action on $\Der_{S|K}$, the module of $K$-linear derivations on $S$, where
\[
g\colon \delta\ \mapsto\ g\circ\delta\circ g^{-1}
\]
for $g\in G$ and $\delta\in\Der_{S|K}$. Using this for $g\colonequals Q\in\O_n(K)$, it follows that
\[
Q\colon \left[\frac{\partial c_k}{\partial x_{ii}}\right]\ \mapsto\ \tilde{Q}\left[\frac{\partial c_k}{\partial x_{ii}}\right],
\]
and hence, taking determinants, that
\[
Q\colon \Delta\ \mapsto\ (\det\tilde{Q})\Delta\ =\ \Delta.
\]
Since $\Delta$ is fixed by the conjugation action of $\O_n(K)$ on $S$, it belongs to the invariant ring $S^{\O_n(K)}$, which is $K[c_1,\dots,c_n]$ by Proposition~\ref{proposition:symmetric:nullcone}.
\end{proof}

\begin{remark}
For $K$ of characteristic two, the expression of $\Delta\colonequals\det\left[\frac{\partial c_k}{\partial x_{ii}}\right]$ as an element of $K[c_1,\dots,c_n]$, for a few small values of $n$, is recorded below; this is solely for the amusement of the authors.
\begin{align*}
n=1 &\qquad \Delta = 1\\
n=2 &\qquad \Delta = c_1\\
n=3 &\qquad \Delta = c_1c_2+c_3\\
n=4 &\qquad \Delta = c_1c_2c_3 + c_1^2c_4 + c_3^2\\
n=5 &\qquad \Delta = c_1c_2c_3c_4 + c_1^2c_4^2 + c_1c_2^2c_5 + c_3^2c_4 + c_2c_3c_5 + c_5^2
\end{align*}
\end{remark}

Versions of the following lemma exist in the literature, e.g.,~\cite[Theorem~27]{Albert} and~\cite[Lemma~1]{BO}; a proof is included here for the convenience of the reader.

\begin{lemma}
\label{lemma:orthogonal:similar}
Let $K$ be an algebraically closed field of characteristic not two.
\begin{enumerate}[\quad\rm(1)]
\item Given a matrix $M\in\GL_n(K)$, there exists a matrix $C\in\GL_n(K)$, that is polynomial in $M$, and satisfies $C^2=M$.

\item Let $A$ and $B$ be symmetric $n\times n$ matrices over $K$. If $A$ and $B$ are similar, then there exists $Q\in\O_n(K)$ with $B=QAQ^\tr$.
\end{enumerate}
\end{lemma}

\begin{proof}
For (1), we may assume that $M$ is in Jordan form, and then that $M$ is a Jordan block, say with nonzero eigenvalue $\lambda$. Note that $(M-\lambda\mathbf{1})^n=0$. Fix a formal power series $f(x)\in K\llbracket x\rrbracket$ with $f(x)^2=x+\lambda$, and set $C\colonequals f(M-\lambda\mathbf{1})$.

For (2), suppose $B=PAP^{-1}$ for $P\in\GL_n(K)$. Then one has
\[
BPP^\tr\ =\ PAP^\tr\ =\ (PAP^\tr)^\tr\ =\ (BPP^\tr)^\tr\ =\ PP^\tr B,
\]
i.e., $PP^\tr$ commutes with $B$. By (1) there exists a matrix $C$, polynomial in $PP^\tr$, with $C^2=PP^\tr$. Note that~$C$ is therefore symmetric and also commutes with~$B$. Set $Q\colonequals C^{-1}P$. Then $Q$ is orthogonal since $QQ^\tr = C^{-1}PP^\tr C^{-1} = \mathbf{1}$. Finally,
\[
BQ\ =\ BC^{-1}P\ =\ C^{-1}BP\ =\ C^{-1}PA\ =\ QA.
\qedhere
\]
\end{proof}

As a consequence of the previous lemma, we have:

\begin{lemma}
\label{lemma:orthogonal:orbits}
Consider the conjugation action of $\SO_n(K)$ on the set of symmetric regular nilpotent matrices over an algebraically closed field $K$ of characteristic other than two. The action is transitive if $n$ is odd; there are two orbits if $n$ is even.
\end{lemma}

\begin{proof}
The conjugation action of $\O_n(K)$ on symmetric regular nilpotent matrices is transitive by Lemma~\ref{lemma:orthogonal:similar}~(2). Let $N$ be a symmetric regular nilpotent matrix; by Remark~\ref{remark:orbits}, the set of $n\times n$ matrices that commute with $N$ is precisely $K[N]\cong K[t]/(t^n)$, so the stabilizer of $N$ under the $\O_n(K)$ action is
\[
K[N]\cap\O_n(K)\ =\ \{\pm\mathbf{1}\},
\]
the only square roots of $1$ in the ring $K[t]/(t^n)$ being $\pm 1$. It follows that the stabilizer of $N$ under the $\SO_n(K)$ action, i.e., $K[N]\cap\SO_n(K)$, is trivial if $n$ is odd, and $\{\pm\mathbf{1}\}$ if $n$ is even.
\end{proof}

\begin{example}
\label{example:char:two}
Lemma~\ref{lemma:orthogonal:similar} does not hold over a field $K$ of characteristic two: the $2\times 2$ Jordan block
\[
\begin{bmatrix}1&1\\ 0&1\end{bmatrix}
\]
is not a square in $\GL_2(K)$ so (1) fails; for (2), taking $\alpha\in K\smallsetminus\{0,1\}$, the symmetric nilpotent matrices
\[
A = \begin{bmatrix}1&1\\ 1&1\end{bmatrix}
\quad\text{ and }\quad
B = \begin{bmatrix}\alpha&\alpha\\ \alpha&\alpha\end{bmatrix}
\]
are similar, but not conjugate under the action of $\O_2(K)$.
\end{example}

Nonetheless, for the purposes of a characteristic-free proof of Proposition~\ref{proposition:symmetric:nullcone}, the following suffices:

\begin{lemma}
\label{lemma:orthogonal:char:two}
Let $A$ be a symmetric $n\times n$ matrix with entries from an algebraically closed field~$K$. If $A$ has distinct eigenvalues, then there exists $Q\in\O_n(K)$ such that~$QAQ^\tr$ is diagonal.
\end{lemma}

\begin{proof}
Since $A$ has distinct eigenvalues, there exists $P\in\GL_n(K)$ such that $B\colonequals PAP^{-1}$ is diagonal. But then, as in the proof of Lemma~\ref{lemma:orthogonal:similar}, $PP^\tr$ commutes with $B$. As $B$ is diagonal with distinct diagonal entries, $PP^\tr$ must be diagonal. Let~$C$ be a diagonal matrix with $C^2=PP^\tr$ and set $Q\colonequals C^{-1}P$. Then $Q$ is indeed orthogonal and
\[
QAQ^\tr\ =\ QAQ^{-1}\ =\ C^{-1}PAP^{-1}C\ =\ C^{-1}BC\ =\ B,
\]
where the last equality holds since $B$ and $C$ commute.
\end{proof}

\begin{remark}
Let $K$ be a field of characteristic $p>0$, and let $R$ be as in Theorem~\ref{theorem:sym}. It follows from part (2) of the theorem that $R$ is not $F$-regular; alternatively, the $a$-invariant of the complete intersection ring $R$ is readily seen to be $a(R)=0$. By part (4), $R$ is not $F$-pure in the case $p=2$.

Suppose that $p$ is odd. For $n=2$, we have $R\cong K[x,y]/(x^2+y^2)$, which is~$F$-pure; however, if $n\ge 3$, we conjecture that $R$ is not $F$-pure.
\end{remark}

\section*{Acknowledgments}

The paper arose from questions raised by Craig Huneke at \emph{RayFest}, University of Nebraska-Lincoln, April 2025; we are grateful to Craig and to Linquan Ma for useful discussions. Some of the results were suggested by examples computed with \texttt{Macaulay2}~\cite{Macaulay2} and \texttt{Magma}~\cite{Magma}; \texttt{ChatGPT-5.5 Pro} was used towards finding suitable references.

\bibliographystyle{amsalpha}

\begin{thebibliography}{BCRV}

\bibitem[Al]{Albert}
A.~A.~Albert, \emph{Symmetric and alternate matrices in an arbitrary field, I}, Trans. Amer. Math. Soc.~\textbf{43} (1938), 386--436.

\bibitem[BCP]{Magma}
W.~Bosma, J.~Cannon, and C.~Playoust, \emph{The Magma algebra system. I. The user language}, J. Symbolic Comput.~\textbf{24} (1997), 235--265.

\bibitem[BCRV]{BCRV}
W.~Bruns, A.~Conca, C.~Raicu, and M.~Varbaro, \emph{Determinants, Gr\"{o}bner bases and cohomology}, Springer Monogr. Math., Springer, Cham, 2022. 

\bibitem[BO]{BO}
D.~Bukov\v{s}ek and M.~Omladi\v{c}, \emph{Linear spaces of symmetric nilpotent matrices}, Linear Algebra Appl.~\textbf{530} (2017), 384--404.

\bibitem[Ca]{Carvajal-Rojas}
J.~A.~Carvajal-Rojas, \emph{Finite torsors over strongly $F$-regular singularities}, \'Epijournal G\'eom. Alg\'ebrique~\textbf{6} (2022), article no. 1, 30~pp.

\bibitem[CM]{Collingwood:McGovern}
D.~H.~Collingwood and W.~M.~McGovern, \emph{Nilpotent orbits in semisimple Lie algebras}, Van Nostrand Reinhold Mathematics Series, Van Nostrand Reinhold Co., New York, 1993.

\bibitem[DK]{Derksen:Kemper}
H.~Derksen and G.~Kemper, \emph{Computational invariant theory}, Encyclopaedia of Mathematical Sciences \textbf{130}, Springer-Verlag, Berlin, 2002.

\bibitem[Do]{Donkin}
S.~Donkin, \emph{The normality of closures of conjugacy classes of matrices}, Invent. Math.~\textbf{101} (1990), 717--736.

\bibitem[Fe]{Fedder}
R.~Fedder, \emph{$F$-purity and rational singularity}, Trans. Amer. Math. Soc.~\textbf{278} (1983), 461--480.

\bibitem[FW]{Fedder:Watanabe}
R.~Fedder and K.-i.~Watanabe, \emph{A characterization of $F$-regularity in terms of $F$-purity}, in: Commutative algebra (Berkeley, CA, 1987), pp.~227--245, Math. Sci. Res. Inst. Publ.~\textbf{15}, Springer, New York, 1989.

\bibitem[Gl]{Glassbrenner}
D.~Glassbrenner, \emph{Strong $F$-regularity in images of regular rings}, Proc. Amer. Math. Soc.~\textbf{124} (1996), 345--353.

\bibitem[GS]{Macaulay2}
D.~R.~Grayson and M.~E.~Stillman, \emph{Macaulay2, a software system for research in algebraic geometry}, available at \url{http://www.math.uiuc.edu/Macaulay2/}.

\bibitem[HH]{HH:strong}
M.~Hochster and C.~Huneke, \emph{Tight closure and strong $F$-regularity}, Mem. Soc. Math. France~\textbf{38} (1989), 119--133.

\bibitem[HJPS]{HJPS}
M.~Hochster, J.~Jeffries, V.~Pandey, and A.~K.~Singh, \emph{When are the natural embeddings of classical invariant rings pure?} Forum Math. Sigma~\textbf{11} (2023), paper no. e67, 43~pp.

\bibitem[KP]{KP}
H.~Kraft and C.~Procesi, \emph{Closures of conjugacy classes of matrices are normal}, Invent.~Math.~\textbf{53} (1979), 227--247.

\bibitem[Lo1]{Lorincz:ASENS}
A.~C.~L\H orincz, \emph{On the collapsing of homogeneous bundles in arbitrary characteristic}, Ann. Sci. \'Ec. Norm. Sup\'er.~(4)~\textbf{56} (2023), 1313--1337.

\bibitem[Lo2]{Lorincz:symplectic}
A.~C.~L\H orincz, \emph{Singularities of orthogonal and symplectic determinantal varieties},\newline
\url{https://arxiv.org/abs/2311.07549}.

\bibitem[MvdK]{Mehta:vdK}
V.~B.~Mehta and W.~van der Kallen, \emph{A simultaneous Frobenius splitting for closures of conjugacy classes of nilpotent matrices}, Compositio Math.~\textbf{84} (1992), 211--221.

\bibitem[PTW]{PTW}
V.~Pandey, Y.~Tarasova, and U.~Walther, \emph{On the natural nullcones of the symplectic and general linear groups}, J. Lond. Math. Soc.~\textbf{111} (2025), paper no.~e70078, 31pp.

\bibitem[Po]{Popov}
V.~L.~Popov, \emph{Picard groups of homogeneous spaces of linear algebraic groups and one-dimensional homogeneous vector fiberings}, Izv. Akad. Nauk SSSR Ser. Mat.~\textbf{38} (1974), 294--322.

\bibitem[Si]{Singh:cyclic}
A.~K.~Singh, \emph{Cyclic covers of rings with rational singularities}, Trans. Amer. Math. Soc.~\textbf{355} (2003), 1009--1024.

\bibitem[Sm]{Smith:ratsing}
K.~E.~Smith, \emph{$F$-rational rings have rational singularities}, Amer. J.~Math.~\textbf{119} (1997), 159--180.

\bibitem[Wa]{Watanabe:dim2}
K.-i. Watanabe, \emph{$F$-regular and $F$-pure normal graded rings}, J.~Pure Appl. Algebra~\textbf{71} (1991), 341--350.

\bibitem[We]{Weyman}
J.~Weyman, \emph{Cohomology of vector bundles and syzygies}, Cambridge Tracts in Math.~\textbf{149}, Cambridge University Press, Cambridge, 2003.

\end{thebibliography}

\end{document}